\documentclass[11pt, leqno]{article}
\usepackage{amsfonts}
\usepackage{mathrsfs}

\usepackage{euscript,amstext,amsmath,amsthm,amssymb,amscd}

\newcommand{\comment}[1]{}

\newtheorem{theorem}{Theorem}[section]
\newtheorem{conjecture}[theorem]{Conjecture}
\newtheorem{proposition}[theorem]{Proposition}
 
\newtheorem{lemma}[theorem]{Lemma}
\theoremstyle{remark}
\newtheorem{remark}[theorem]{Remark}
\theoremstyle{definition}

\title{Lower bounds of Lipschitz constants on foliations}
\author{Guangxiang Su\footnote{Chern Institute of Mathematics \& LPMC, Nankai University,
Tianjin 300071, P.R. China.(guangxiangsu@nankai.edu.cn)\ Supported by NSFC 11571183.}}
\date{}

\begin{document}

\maketitle

\begin{abstract}
In this paper we consider Llarull's theorem in the foliation case and get a lower bound of the Lipschitz constant of the map $M\to S^n$ in the foliation case under the spin condition. 
\end{abstract}

\maketitle
\renewcommand{\theequation}{\thesection.\arabic{equation}}
\setcounter{equation}{0}

\section{Introduction}
\setcounter{equation}{0}

In \cite{G}, M. Gromov conjectured the following. 
\begin{conjecture}(Gromov)
Let $g$ be a Riemannian metric on $S^n$ such that $g\geq g_0$ where $g_0$ is the standard metric of constant curvature. Then the scalar curvature $k_g$ must become small somewhere, more precisely, $\inf k_g\leq c(n)k_{g_0}$, where $c(n)\leq 1$ is a constant that depends on the dimension $n$, with best constant when $c(n)=1$.
\end{conjecture}

A map $f: M\to N$ between Riemannian manifolds is said to be $\varepsilon$-contracting if $\|f_* v\|\leq \varepsilon \|v\|$ for tangent vectors $v$ to $M$ (cf. \cite{GL}).

The normalized scalar curvature of a manifold $M$ of dimension $n$ is defined as
$$\widetilde{k}={k\over{n(n-1)}},$$
where $k$ is the usual scalar curvature.

In \cite{L}, M. Llarull proved the following theorem which confirmed Gromov's conjecture. 

\begin{theorem}\label{thmL}(\cite{L})
Let $M$ be a compact Riemannian spin manifold of dimension $n$. Suppose there exists a $1$-contracting map $f:(M,g)\to (S^n, g_0)$ of non-zero degree. Then either there exists $x\in M$ with $\widetilde{k}_g (x)<1$, or $M\equiv S^n$ and $f$ is an isometry. 
\end{theorem}

Recall that for a map $f:M\to S^n$ the Lipschitz constant (cf. \cite{G2, GL}) is defined by 
\begin{align}
Lip(f)=\sup_{x_1\neq x_2}{{dist_{S^n}(f(x_1),f(x_2))}\over{dist_M(x_1,x_2)}}.
\end{align}
In \cite[Section 3]{G2}, Gromov pointed out that Theorem \ref{thmL} is related to the problems of the Lipschitz constants of the maps $M\to S^n$.

Let $F$ be an integrable subbundle of the tangent vector bundle $TM$ of a closed smooth manifold $M$. Let $g^F$ be a metric on $F$, and $k^F_g\in C^\infty (M)$ be the associated leafwise scalar curvature (cf. \cite[(0.1)]{Z1}). Let $\widetilde{k}^F_g$ be the normalized leafwise scalar curvature, i.e.
$$\widetilde{k}^F_g={{k^F_g}\over{{\rm dim}F({\rm dim}F-1)}}.$$

In this paper, we prove the following theorem which partly generalize Theorem \ref{thmL} in the foliation case.

\begin{theorem}\label{thm1.1}

Let $M$ be a closed Riemannian manifold of dimension $n$ and $F$ be a foliation on $M$. Suppose $TM$ or $F$ is spin and there exist a smooth map $f:(M,g)\to (S^n, g_0)$ of non-zero degree such that for any $v\in \Gamma(F)$, $\|f_*(v)\| \leq \|v\|$. Then there exists $x\in M$ with $\widetilde{k}^F_g (x)\leq 1$. 

\end{theorem}

From Theorem \ref{thm1.1}, one sees that the leafwise scalar curvature is also related to the lower bounds of the Lipschitz constants. If $\widetilde{k}^F_g>1$, then there exists $v\in \Gamma(F)$, such that $\|f_*(v)\|> \|v\|$. So from Theorem \ref{thm1.1}, one has the following theorem. 

\begin{theorem}
Let $M$ be a closed Riemannian manifold of dimension $n$ and $F$ be a foliation on $M$. Suppose $TM$ or $F$ is spin and $\widetilde{k}^F_g>1$, then for  smooth maps $f:(M,g)\to (S^n, g_0)$ of non-zero degree, we have $Lip(f)> 1$. 
\end{theorem}

Our proof of Theorem \ref{thm1.1} combines the methods in \cite{L}, \cite{Z1} and \cite{Z2}. It is based on deforming (twisted) sub-Dirac operators on the Connes fibration. It will be carries out in Section 2.

\section{Proof of Theorem \ref{thm1.1}}
\setcounter{equation}{0}
In this section, we give a proof of the main theorem. We give the details for the case $TM$ is spin, $F$ is spin is similarly.

If there does not exist any point such that $\widetilde{k}^F_g\leq 1$, then there exits $\delta>0$ such that $\widetilde{k}^F_g(x)-1\geq \delta$, for any $x\in M$.

\subsection{The dimension of $M$ is even}
Over $(S^{2n}, g_0)$, we have the spinor bundle (cf. \cite{LM})
\begin{align}
E_0=P_{Spin_{2n}}\left(S^{2n}\right)\times_{\lambda}\mathbb{C}l_{2n},
\end{align}
with the induced metric and connection from $(S^{2n},g_0)$. Fix $x\in S^{2n}$ and choose local point wise orthonormal tangent vector fields around $x$, $\{\varepsilon_1, \varepsilon_2,\dots,\varepsilon_{2n}\}$ such that $(\nabla \varepsilon_k)_x=0$. Let $\omega_0$ 
$$\omega_0=i^n \varepsilon_1\cdot \varepsilon_2\cdots \varepsilon_{2n}.$$
Then $\omega_0$ gives the splitting 
$$E_0=E_0^{+}\oplus E_0^{-}$$
into the $+1$ and $-1$ eigenspaces of $\omega_0$.

Fix $p\in M$. Let $\{f_1,\dots, f_{2n}\}$ be a $g$-orthonormal tangent frame near $p\in M$ such that 
$(\nabla f_k)_p=0$ for each $k$. Let $\{e_1,\dots,e_{2n}\}$ be a $g_0$-orthonormal tangent frame near $f(p)\in S^{2n}$ such that $(\nabla e_k)_{f(p)}=0$ for each $k$. Moreover, the bases $\{f_1,\dots,f_{2n}\}$ and $\{e_1,\dots,e_{2n}\}$ can be chosen so that $e_j=\lambda_j f_* f_j$ for appropriate $\{\lambda_j\}_{j=1}^{2n}$.

Assume $\{f_j\}, 1\leq j\leq {\rm dim}F$, is a base of $\Gamma(F)$. By assumption, we have $\lambda_j\geq 1$, $1\leq j\leq {\rm dim}F$.

Following \cite[Section 5]{C} (cf. \cite[Section 2.1]{Z1}), let $\pi:\mathcal{M}\to M$ be the Connes fibration over $M$ such that for any $x\in M$, $\mathcal{M}_x=\pi^{-1}(x)$ is the space of Euclidean metrics on the linear space $T_x M/F_x$. Let $T^V \mathcal{M}$ denote the vertical tangent bundle of the fibration $\pi: \mathcal{M}\to M$. Then it carries a natural metric $g^{T^V \mathcal{M}}$.

By using the Bott connection on $TM/F$, which is leafwise flat, one lifts $F$ to an integrable subbundle $\mathcal{F}$ of $T\mathcal{M}$. Then $g^F$ lifts to a Euclidean metric $g^{\mathcal{F}}=\pi^* g^F$ on $\mathcal{F}$.

Let $\mathcal{F}^{\bot}_{1}\subset T\mathcal{M}$ be a subbundle, which is transversal to $\mathcal{F}\oplus T^V \mathcal{M}$, such that we have a splitting $T\mathcal{M}=(\mathcal{F}\oplus T^V \mathcal{M})\oplus \mathcal{F}^{\bot}_{1}$. Then $\mathcal{F}^{\bot}_{1}$ can be identified with $T\mathcal{M}/(\mathcal{F}\oplus T^V \mathcal{M})$ and carries a canonically induced metric $g^{\mathcal{F}^{\bot}_{1}}$. We denote $\mathcal{F}^{\bot}_{2}=T^V \mathcal{M}$.

Set $E=f^* E_0$. Let $\mathcal{E}=\pi^* E$ be the lift of $E$ which carries the lifted Hermitian metric $g^{\mathcal{E}}=\pi ^* g^E$ and the lifted Hermitian connection $\nabla^{\mathcal{E}}=\pi^* \nabla^E$. Let $R^{\mathcal{E}}=(\nabla^{\mathcal{E}})^2$ be the curvature of $\nabla^{\mathcal{E}}$.

For any $\beta, \varepsilon>0$, following \cite[(2.15)]{Z1}, let $g_{\beta,\varepsilon}^{T\mathcal{M}}$ be the metric on $T\mathcal{M}$ defined by the orthogonal splitting,
\begin{align}
T\mathcal{M}=\mathcal{F}\oplus \mathcal{F}^{\bot}_{1}\oplus \mathcal{F}^{\bot}_{2},\ g^{T\mathcal{M}}_{\beta,\varepsilon}=\beta^2 g^{\mathcal{F}}\oplus {{g^{\mathcal{F}^{\bot}_{1}}}\over{\varepsilon^2}}\oplus g^{\mathcal{F}^{\bot}_{2}}. 
\end{align}

Now we replace the sub-Dirac operator constructed in \cite[(2.16)]{Z1} by the obvious twisted (by $\mathcal{E}$) analogue (\cite[(1.3)]{Z2})

\begin{multline}
D^{\mathcal{E}}_{\mathcal{F}\oplus \mathcal{F}^{\bot}_{1},\beta,\varepsilon}:\Gamma\left(S_{\beta,\varepsilon}\left(\mathcal{F}\oplus \mathcal{F}^{\bot}_{1}\right)\widehat{\otimes}\Lambda^*\left(\mathcal{F}^{\bot}_{2}\right)\otimes \mathcal{E}\right) \\
\to \Gamma\left(S_{\beta,\varepsilon}\left(\mathcal{F}\oplus \mathcal{F}^{\bot}_{1}\right)\widehat{\otimes}\Lambda^*\left(\mathcal{F}^{\bot}_{2}\right)\otimes \mathcal{E}\right).
\end{multline}

Take a metric on $TM/F$. This is equivalent to taking an embedded section $s:M\hookrightarrow\mathcal{M}$ of the Connes fibration $\pi: \mathcal{M}\to M$. Then we have a canonical inclusion $s(M)\subset \mathcal{M}$.

For any $p\in\mathcal{M}\setminus s(M)$, we connect $p$ and $s(\pi(p))\in s(M)$ by the unique geodesic in $\mathcal{M}_{\pi(p)}$. Let $\sigma(p)\in\mathcal{F}^{\bot}_{2}|_{p}$ denote the unit vector tangent to this geodesic. Let $\rho(p)=d^{\mathcal{M}_{\pi(p)}}(p,s(\pi(p)))$ denote the length of this geodesic.

Let $\mathfrak{f}:[0,1]\to [0,1]$ be a smooth function such that $\mathfrak{f}(t)=0$ for $0\leq t\leq {1\over 4}$, while $\mathfrak{f}(t)=1$ for ${1\over 2}\leq t\leq 1$. Let $h:[0,1]\to [0,1]$ be a smooth function such that $h(t)=1$ for $0\leq t\leq {3\over 4}$, while $h(t)=0$ for ${7\over 8}\leq t\leq 1$. 

For any $R>0$, denote
\begin{align}
\mathcal{M}_{R}=\left\{p\in\mathcal{M}: \rho(p)\leq R\right\}.
\end{align}
Then $\mathcal{M}_R$ is a smooth manifold with boundary.

On the other hand, the following formula holds on $\mathcal{M}_R$ (cf. \cite[(2.28)]{Z1}, \cite[(1.4)]{Z2})
\begin{multline}\label{a2.30}
\left(D^{\mathcal{E}}_{\mathcal{F}\oplus \mathcal{F}^{\bot}_{1},\beta,\varepsilon}\right)^2=-\Delta^{\mathcal{E},\beta,\varepsilon}+{{k^{\mathcal{F}}}\over{4\beta^2}}+{1\over{2\beta^2}}\sum_{i,j=1}^{{\rm rk}F}R^{\mathcal{E}}(f_i,f_j)c_{\beta,\varepsilon}(\beta^{-1}f_i)c_{\beta,\varepsilon}(\beta^{-1}f_j)\\+\mathcal{O}_{R}\left({1\over \beta}+{\varepsilon^2\over \beta^2}\right),
\end{multline}
where $-\Delta^{\mathcal{E},\beta,\varepsilon}\geq 0$ is the corresponding Bochner Laplacian, $k^{\mathcal{F}}=\pi^* k^F$ and $f_1,\cdots, f_{{\rm rk}F}$ is an orthonormal basis of $(\mathcal{F},g^{\mathcal{F}})$.

Since \cite[Lemma 4.3]{L} and \cite[Lemma 4.5]{L} hold for fixed $(i,j)$, 
proceeding as the computations in \cite[Lemmas 4.3, 4.5]{L}, for any $\phi\in \Gamma(S_{\beta,\varepsilon}(\mathcal{F}\oplus \mathcal{F}^{\bot}_{1})\widehat{\otimes}\Lambda^*(\mathcal{F}^{\bot}_{2})\otimes\mathcal{E})$ supported in $\mathcal{M}_R$,
one has
\begin{multline}\label{add2.6}
\left\langle{1\over{2\beta^2}}\sum_{i,j=1}^{{\rm rk}F}R^{\mathcal{E}}(f_i,f_j)c_{\beta,\varepsilon}(\beta^{-1}f_i)c_{\beta,\varepsilon}(\beta^{-1}f_j)\phi,\phi\right\rangle\\
\geq -{1\over{4\beta^2}}{\rm dim}F({\rm dim}F-1)\|\phi\|^2.
\end{multline}
Then by (\ref{a2.30}) and (\ref{add2.6}), for any $\phi\in \Gamma(S_{\beta,\varepsilon}(\mathcal{F}\oplus \mathcal{F}^{\bot}_{1})\widehat{\otimes}\Lambda^*(\mathcal{F}^{\bot}_{2})\otimes\mathcal{E})$ supported in $\mathcal{M}_R$, one gets
\begin{multline}\label{a2.31}
\left\langle  \left(D^{\mathcal{E}}_{\mathcal{F}\oplus \mathcal{F}^{\bot}_{1},\beta,\varepsilon}\right)^2 \phi,\phi\right\rangle\\
\geq {1\over {4\beta^2}}{\rm dim}F ({\rm dim}F-1)\left(\widetilde{k}^{\mathcal{F}}_g-1\right)\|\phi\|^2+\mathcal{O}_{R}\left({1\over \beta}+{\varepsilon^2\over \beta^2}\right)\|\phi\|^2. 
\end{multline}

From (\ref{a2.31}), proceeding as the proof of \cite[Lemma 2.4]{Z1}, one can get the following analogue inequality of \cite[(2.22)]{Z1}.

\begin{lemma}\label{l2.1}
There exist $C_0$, $R_0>0$, such that for any (fixed) $R\geq R_0$, when $\beta$, $\varepsilon>0$ (which may depend on $R$) are small enough, for any $\phi\in \Gamma(S_{\beta,\varepsilon}(\mathcal{F}\oplus \mathcal{F}^{\bot}_{1})\widehat{\otimes}\Lambda^*(\mathcal{F}^{\bot}_{2})\otimes\mathcal{E})$ supported in $\mathcal{M}_R$, one has 
\begin{align}\label{a2.24}
\left\|\left(D^{\mathcal{E}}_{\mathcal{F}\oplus \mathcal{F}^{\bot}_{1},\beta,\varepsilon}+{{\mathfrak{f}\left({\rho\over R}\right)\widehat{c}(\sigma})\over{\beta}}\right)\phi\right\|\geq {C_0\sqrt{\delta}\over \beta}\|\phi\|.
\end{align}
\end{lemma}

Next we recall the construction of the operator $P^{\mathcal{E}}_{R,\beta,\varepsilon}$ from \cite{Z1}.

Let $\partial \mathcal{M}_R$ bound another oriented manifold $\mathcal{N}_R$ so that $\widetilde{\mathcal{N}}_R=\mathcal{M}_R\cup \mathcal{N}_R$ is a closed manifold. Let $H$ be a Hermitian vector bundle over $\mathcal{M}_R$ such that $(S_{\beta,\varepsilon}(\mathcal{F}\oplus \mathcal{F}^{\bot}_{1})\widehat{\otimes}\Lambda^*(\mathcal{F}^{\bot}_{2})\otimes \mathcal{E})_{+}\oplus H$ is a trivial vector bundle near $\partial \mathcal{M}_R$, under the identification $\widehat{c}(\sigma)+{\rm Id}_{H}$.

By obviously extending the above trivial vector bundles to $\mathcal{N}_R$, we get a $\mathbb{Z}_2$-graded Hermitian vector bundle $\xi=\xi_+\oplus \xi_-$ over $\widetilde{\mathcal{N}}_R$ and an odd self-adjoint endomorphism $V=v+v^*\in\Gamma({\rm End}(\xi))$ (with $v:\Gamma(\xi_+)\to \Gamma(\xi_-), v^*$ being the adjoint of $v$) such that 
\begin{align}
\xi_{\pm}=(S_{\beta,\varepsilon}(\mathcal{F}\oplus \mathcal{F}^{\bot}_{1})\widehat{\otimes}\Lambda^*(\mathcal{F}^{\bot}_{2})\otimes \mathcal{E})_{\pm}\oplus H
\end{align}
over $\mathcal{M}_R$, $V$ is invertible on $\mathcal{N}_R$ and
\begin{align}\label{a2.8}
V=\mathfrak{f}\left({\rho}\over R\right)\widehat{c}(\sigma)+{\rm Id}_{H}
\end{align}
on $\mathcal{M}_R$, which is invertible on $\mathcal{M}_R\setminus \mathcal{M}_{{R}\over{2}}$.

Recall that $h({\rho\over R})$ vanishes near $\partial \mathcal{M}_R$. We extend it to a function on $\widetilde{\mathcal{N}}_R$ which equals to zero on $\mathcal{N}_R$, and we denote the resulting function on $\widetilde{\mathcal{N}}_R$ by $\widetilde{h}_R$. Let $\pi_{\widetilde{\mathcal{N}}_R}:T\widetilde{\mathcal{N}}_R\to \widetilde{\mathcal{N}}_R$ be the projection of the tangent bundle of $\widetilde{\mathcal{N}}_R$. Let $\gamma^{\widetilde{\mathcal{N}}_R}\in {\rm Hom}(\pi^*_{\widetilde{\mathcal{N}}_R}\xi_+,\pi^*_{\widetilde{\mathcal{N}}_R}\xi_-)$ be the symbol defined by 
\begin{align}\label{a2.9}
\gamma^{\widetilde{\mathcal{N}}_R}(p,w)=\pi^*_{\widetilde{\mathcal{N}}_R}\left(\sqrt{-1}\widetilde{h}^2_R c_{\beta,\varepsilon}(w)+v(p)\right)\ {\rm for}\ p\in\widetilde{\mathcal{N}}_R,\ w\in T_p \widetilde{\mathcal{N}}_R. 
\end{align}
By (\ref{a2.8}) and (\ref{a2.9}), $\gamma^{\widetilde{\mathcal{N}}_R}$ is singular only if $w=0$ and $p\in \mathcal{M}_R$. Thus $\gamma^{\widetilde{\mathcal{N}}_R}$ is an elliptic symbol.

On the other hand, it is clear that $\widetilde{h}_R D^{\mathcal{E}}_{\mathcal{F}\oplus \mathcal{F}^{\bot}_1,\beta,\varepsilon}\widetilde{h}_R$ is well defined on $\widetilde{\mathcal{N}}_R$ if we define it to equal to zero on $\widetilde{\mathcal{N}}_R\setminus \mathcal{M}_R$.

Let $A:L^2(\xi)\to L^2(\xi)$ be a second order positive elliptic differential operator on $\widetilde{\mathcal{N}}_R$ preserving the $\mathbb{Z}_2$-grading of $\xi=\xi_+\oplus \xi_{-}$, such that  its symbol equals to $|\eta|^2$ at $\eta\in T\widetilde{\mathcal{N}}_R$. Let $P^{\mathcal{E}}_{R,\beta,\varepsilon}:L^2(\xi)\to L^2(\xi)$ be the zeroth order pseudodifferential operator on $\widetilde{\mathcal{N}}_R$ defined by
\begin{align}
P^{\mathcal{E}}_{R,\beta,\varepsilon}=A^{-{1\over 4}}\widetilde{h}_R D^{\mathcal{E}}_{\mathcal{F}\oplus \mathcal{F}^{\bot}_1,\beta,\varepsilon}\widetilde{h}_R A^{-{1\over 4}}+{V\over \beta}. 
\end{align}
Let $P^{\mathcal{E}}_{R,\beta,\varepsilon,+}:L^2(\xi_{+})\to L^2(\xi_{-})$ be the obvious restriction. Moreover, the analogue of \cite[(2.34)]{Z1} now takes the form
\begin{align}\label{2.6}
{\rm ind}\left(P^{\mathcal{E}}_{R,\beta,\varepsilon,+}\right)=\left\langle \widehat{A}(TM){\rm ch}(E),[M]\right\rangle.
\end{align}

For any $0\leq t\leq 1$, set
\begin{align}
P^{\mathcal{E}}_{R,\beta,\varepsilon,+}(t)=A^{-{1\over 4}}\widetilde{h}_R D^{\mathcal{E}}_{\mathcal{F}\oplus \mathcal{F}^{\bot}_1,\beta,\varepsilon}\widetilde{h}_R A^{-{1\over 4}}+{{tV}\over{\beta}}+A^{-{1\over 4}}{{(1-t)V}\over{\beta}}A^{-{1\over 4}}.
\end{align}
Then $P^{\mathcal{E}}_{R,\beta,\varepsilon,+}(t)$ is a smooth family of zeroth order pseudodifferential operators such that the corresponding symbol $\gamma(P^{\mathcal{E}}_{R,\beta,\varepsilon,+}(t))$ is elliptic for $0<t\leq 1$. Thus $P^{\mathcal{E}_{R,\beta,\varepsilon,+}}(t)$ is a continuous family of Fredholm operators for $0<t\leq 1$ with $P^{\mathcal{E}}_{R,\beta,\varepsilon,+}(1)=P^{\mathcal{E}}_{R,\beta,\varepsilon,+}$.

By Lemma \ref{l2.1} and \cite[Lemma 2.4(ii)]{Z1}, proceeding as the proof of \cite[Proposition 2.5]{Z1}, one has the following proposition. 

\begin{proposition}\label{p2.2}
There exist $R$, $\beta$, $\varepsilon>0$ such that the following identity holds:
\begin{align}
{\rm dim}\left({\rm ker}\left(P^{\mathcal{E}}_{R,\beta,\varepsilon,+}(0)\right)\right)={\rm dim}\left({\rm ker}\left(P^{\mathcal{E}}_{R,\beta,\varepsilon,+}(0)^*\right)\right)=0.
\end{align}
\end{proposition}

By (\ref{2.6}), Proposition \ref{p2.2} and \cite[Theorem 0.1]{Z1}, one has
\begin{multline}\label{2.7}
0=\left\langle \widehat{A}(TM){\rm ch}(E),[M]\right\rangle={\rm rk}(E_0)\widehat{A}(M)\\+\left\langle \widehat{A}(TM)f^*\left({\rm ch}(E_0)-{\rm rk}(E_0)\right),[M]\right\rangle
={\rm deg}(f)\left\langle {\rm ch}(E_0),S^{2n}\right\rangle,
\end{multline}
which contradicts with ${\rm deg}(f)\left\langle {\rm ch}(E_0),S^{2n}\right\rangle\neq 0$.

\subsection{The dimension of $M$ is odd}
Let $M$ be a compact spin manifold of dimension $2n-1$, with Riemannian metric $g$. Let $S_r ^{2n-1}$ be $(2n-1)$-sphere of radius $r$ with the standard metric $g_0$. Let $F$ be a foliation on $M$. Let $f:M\to S^{2n-1}$ be a map of non-zero degree and $f_{*}|_{\Gamma(F)}$ is $1$-contracting.

Consider 
\begin{align}
M\times S_r^1\xrightarrow{f\times {1\over r}id}S^{2n-1}\times S^1\xrightarrow{\mathfrak{h}}S^{2n-1}\wedge S^1 \cong S^{2n},
\end{align}
where $S^1_r$ is the one-dimensional sphere of radius $r$, $f\times {1\over r}id$ is defined as $(f\times {1\over r}id)(p,t)=(f(p), {t\over r})$, $(p,t)\in M\times S^1$, and $\mathfrak{h}$ is map of non-zero degree. $\mathfrak{h}|_{F\times S^1}$ is $1$-contracting.

Consider now the following metric. On $M\times S^1_r$, $g+ds^2$ where $ds$ is the standard metric on $S^1_r$; on $S^{2n-1}\times S^1$, $g_0+ds^2$ where $ds$ is the standard metric on $S^1$; and on $S^{2n}$, $\widetilde{g}$ is the standard metric on the unit sphere.

The compose map $\widetilde{f}=\mathfrak{h}\circ (f\times {1\over r}id)$ is of non-zero degree form $M^{2n-1}\times S^1\to S^{2n}$. $\widetilde{f}|_{F\times S^1}$ is also $1$-contracting, for $v\in \Gamma(F)$,
\begin{align}
\left\|\widetilde{f}_* (v,t)\right\|=\left\|\mathfrak{h}_*\left(f_* v,{t\over r}\right)\right\|\leq \|f_* v\|+\left\|{t\over r}\right\|
\leq \|v\|+{1\over r}\|t\|\leq \|v\|+\|t\|.
\end{align}
We assume $r>1$.

We can now apply the same method used for the even-dimensional case. Construct complex spinor bundles $S$ over $M^{2n-1}\times S^1_r$ and $E_0$ over $S^{2n}$, respectively; and consider the bundle $S\otimes E$ over $M^{2n-1}\times S^1_r$, where $E=\widetilde{f}^* E_0$.

Choose a basis $\{f_1,\dots,f_{2n-1},f_{2n}\}$ of $(g+ds^2)$-orthonormal adapted tangent vectors around $x\in M^{2n-1}\times S^1_r$ such that $(\nabla f_k)_x=0$ for each $k$ and such that $f_1,\dots,f_{2n-1}$ are tangent to $M^{2n-1}$ and $e_{2n}$ is tangent to $S^1_r$. As before choose $\widetilde{g}$-orthonormal basis $\{e_1,\dots,e_{2n}\}$ around $\widetilde{f}(x)$ in $S^{2n}$. We assume that $f_i$, $1\leq i\leq {\rm dim}F$ is a basis of $F$.

Therefore, we can find positive scalars $\{\lambda_i\}_{i=1}^{2n}$ such that $e_i=\lambda_i \widetilde{f}_* f_i$. Then we have that
$$1=\widetilde{g}\left(e_i,e_j\right)=\widetilde{g}\left(\lambda_i \widetilde{f}_* f_i, \lambda_i \widetilde{f}_* f_i\right)=\lambda^2_i \widetilde{g}\left(\widetilde{f}_*f_i,\widetilde{f}_*f_i\right),$$
thus for $1\leq i\leq {\rm dim}F$,
\begin{align}
1=\lambda^2_i \widetilde{g}\left(\widetilde{f}_*f_i,\widetilde{f}_* f_i\right)\leq \lambda_i^2 g_0 \left(f_* f_i,f_* f_i\right)\leq \lambda^2_i g(e_i,e_i)=\lambda_i^2
\end{align}
and $1\leq \lambda^2_i$.

For $i=2n$,
$$1=\lambda_{2n}^2\widetilde{g}\left(\widetilde{f}_* f_{2n},\widetilde{f}_{*}f_{2n}\right)\leq \lambda^2_{2n}ds^2 \left({f_{2n}\over r},{f_{2n}\over r}\right).$$
Then 
$$r^2\leq \lambda^2_{2n}. $$
In this case, (\ref{a2.31}) is replaced by
\begin{multline}\label{2.14}
\left\langle  \left(D^{\mathcal{E}}_{\mathcal{F}\oplus \mathcal{F}^{\bot}_{1},\beta,\varepsilon}\right)^2 \phi,\phi\right\rangle\\
\geq {1\over {4\beta^2}}{\rm dim}F ({\rm dim}F-1)\left(\widetilde{k}^\mathcal{F}_g-1-{2\over{({\rm dim}F-1)r}}\right)\|\phi\|^2+\mathcal{O}_{R,r}\left({1\over \beta}+{\varepsilon^2\over \beta^2}\right)\|\phi\|^2. 
\end{multline}

If $\widetilde{k}^{F}_g>1$, since (\ref{2.14}) is valid for all $r>1$, one also gets ${\rm ind}\left(P^{\mathcal{E}}_{R,\beta,\varepsilon,+}\right)=0$.
But the Atiyah-Singer index theorem gives (see (\ref{2.6}), (\ref{2.7}))
$${\rm ind}\left(P^{\mathcal{E}}_{R,\beta,\varepsilon,+}\right)\neq 0.$$

As in \cite[Section 2.5]{Z1}, the same proof applies for the case where $F$ is spin, with an obvious modification of the (twisted) sub-Dirac operators (cf. \cite[(2.58)]{Z1}), using \cite[Theorem 0.2]{C} in (\ref{2.7}).

\begin{remark}(cf. \cite{GL}, \cite{L})
Recall that a map $f:M\to N$ between Riemannian manifold is $(\varepsilon,\Lambda^k )$-contracting if 
\begin{align}
\|f^*\alpha\|\leq \varepsilon \|\alpha\|,\ \ \alpha\in \Lambda^k (N).
\end{align}
Note that ``1-contracting" means $(1,\Lambda^1)$-contracting. 
\end{remark}

We have the following immediate consequence.

\begin{theorem}\label{thm2.4}
Let $M$ be a closed Riemannian manifold of dimension $n$ and $F$ be a foliation on $M$. Suppose $TM$ or $F$ is spin and there exists a smooth map $f:(M,g)\to (S^n, g_0)$ of non-zero degree such that for any $v,w\in \Gamma(F)$, $\|f_* v\wedge f_* w\|\leq \|v\wedge w\|$. Then there exits $x\in M$ with $\widetilde{k}^F_g(x)\leq 1$. 
\end{theorem}

\begin{proof}
It follows form the proof of Theorem \ref{thm1.1}. We only need to point out that $\{\lambda_i\}_{i=1}^{{\rm dim}F}$ satisfy 
\begin{multline}
1=\|e_i\wedge e_j\|_{g_0}=\|\lambda_i f_* f_i\wedge \lambda_j f_* f_j\|_{g_0}=\lambda_i \lambda_j \|f_*(f_i\wedge f_j)\|_{g_0}\\
\leq \lambda_i \lambda_j \|f_i\wedge f_j\|_g=\lambda_i\lambda_j.
\end{multline}
Thus $\lambda_i\lambda_j\geq 1$. Then proceeding as the proof of Theorem \ref{thm1.1}, we get Theorem \ref{thm2.4}. 
\end{proof}

\bibliographystyle{amsplain}

\end{document}